\documentclass{article}
\usepackage{amsmath,amssymb,amsthm}

\numberwithin{equation}{section}
\newtheorem{thm}{Theorem}[section]
\newtheorem{lem}[thm]{Lemma}
\newtheorem{prop}[thm]{Proposition}
\newtheorem{cor}[thm]{Corollary}

\theoremstyle{definition}
\newtheorem{defn}{Definition}[section]

\newcommand{\s}{\mathfrak{S}}

\newcommand{\W}{\mathcal{W}}

\newcommand{\ee}{\end{equation}}
\newcommand{\be}[1]{\begin{equation}\label{#1}}
\newcommand{\norm}[1]{\left\Vert#1\right\Vert ^2}
\newcommand{\nJ}{\norm{\nabla J}}
\newcommand{\n}{\nabla}
\newcommand{\nn}{\nabla'}

\newcommand{\ie}{i.~e. }

\newcommand{\thmref}[1]{Theorem~\ref{#1}}
\newcommand{\propref}[1]{Proposition~\ref{#1}}
\newcommand{\lemref}[1]{Lemma~\ref{#1}}

\begin{document}

\centerline{\large \textbf{CANONICAL CONNECTION ON
QUASI-K\"AHLER}}

\smallskip

\centerline{\large \textbf{MANIFOLDS WITH NORDEN METRIC}}

 \vskip 5mm

\centerline{DIMITAR MEKEROV} \vskip 5mm

\begin{abstract}
We study the geometry of the canonical connection on a quasi-K\"ahler
manifold with Norden metric. We consider the cases when the canonical
connection has K\"ahler curvature tensor and parallel torsion, and derive
conditions for an isotropic-K\"ahler manifold.
We give the relation
between the canonical connection, the $B$-connection, and the connection with totally skew symmetric
torsion on quasi-K\"ahler manifolds with Norden
metric.
\par
{\footnotesize {\it Mathematics Subject Classification (2000).}
Primary 53C15, 53C50; Secondary 32Q60, 53B05. \par {\it Key
words.} Norden metric, almost complex manifold, quasi-K\"ahler
manifold, indefinite metric, canonical connection,
parallel torsion.}
\end{abstract}



\section*{Introduction}

A fundamental fact about almost complex manifolds with Hermitian
metric is that the action of the almost complex structure on the
tangent space is isometry. There is another kind of metric, called
a Noden metric, or a $B$-metric, on an almost complex manifold,
such that the action of the almost complex strucutre is
anti-isometry with respect to the metric. Such a manifold in
\cite{GrMeDj} is called a generalized $B$-manifold, while in the
present paper we adopt the more widely used term of almost complex
manifold with Norden metric \cite{GaBo}.

It is known \cite{Lih-1,Lih-2,GrayBNV} that on an almost Hermitian
manifold there exists a unique linear connection $\nn$ with a
torsion $T$ such that $\nn J=\nn g=0$ and $T(x,Jy)=T(Jx,y)$, for
all vector fields $x,y$ on $M$. The group of the conformal
transformations on the metric $g$ generates the conformal group of
transformations of $\nn$. An analogous problem for almost complex
manifolds with Norden metric is treated in \cite{GaMi}, where a
canonical connection on such a manifold is obtained and proved
unique.

In the present paper we consider this canonical connection on an
almost complex manifold with Norden metric in the case when it is
a quasi-K\"ahler manifold.
We establish that the manifold is isotropic-K\"ahler
manifold with Norden metric iff the scalar curvatures of the
canonical connection and the Levi-Civita connection are equal. We
consider the cases of the canonical connection with K\"ahler
curvature tensor and parallel torsion. We give the relation
between the canonical connection, the $B$-connection
(\cite{Mek-1}), and the connection with totally skew symmetric
torsion (\cite{Mek-2}) on quasi-K\"ahler manifolds with Norden
metric.
The $B$-connection is an analogue of the first canonical
Lihnerowicz connection in the Hermitian geometry
\cite{Lih-3,Ya,Gaud}.
The connection with a totally skew-symmetric torsion is
known as a Bismut connection or a $KT$-connection. It is applied
in mathematics as well as in theoretical physics. For instance,
the local index theorem for non-K\"ahler Hermitian manifolds is
proved in \cite{Bis} using this connection. Moreover, this
connection is applied in string theory \cite{Stro}.


\section{Preliminaries}

Let $(M,J,g)$ be a $2n$-dimensional \emph{almost complex manifold
with Norden metric}, i.~e.
\begin{equation}\label{2.1}
J^2x=-x, \qquad g(Jx,Jy)=-g(x,y),
\end{equation}
for all differentiable vector fields $x$, $y$ on $M$. The
\emph{associated metric} $\tilde{g}$ of $g$ on $M$, given by
$\tilde{g}(x,y)=g(x,Jy)$, is a Norden metric, too. The signature
of both metrics is necessarily $(n,n)$.

Further, $x$, $y$, $z$, $w$ will stand for arbitrary
differentiable vector fields on $M$ (or vectors in the tangent
space  of $M$ at an arbitrary point $p\in M$).

The Levi-Civita connection of $g$ is denoted by $\nabla$. The
tensor filed $F$ of type $(0,3)$ on $M$ is defined by
\begin{equation}\label{2.2}
F(x,y,z)=g\bigl( \left( \nabla_x J \right)y,z\bigr).
\end{equation}
It has the following properties \cite{GrMeDj}:
\begin{equation}\label{2.3}
F(x,y,z)=F(x,z,y)=F(x,Jy,Jz),\quad F(x,Jy,z)=-F(x,y,Jz).
\end{equation}

In \cite{GaBo}, the considered manifolds are classified into eight classes with respect to $F$.
The class $\W_0$ of the
\emph{K\"ahler manifolds with Norden metric} belongs to any
of the other seven classes. It is determined by the condition $F(x,y,z)=0$, which is
equivalent to $\n J=0$.

The condition                                                                        
\begin{equation}\label{2.4}
\mathop{\s} \limits_{x,y,z} F(x,y,z)=0,
\end{equation}
where $\mathop{\s} \limits_{x,y,z}$ is the cyclic sum over $x,y,z$,
characterizes the class $\W_3$ of the \emph{quasi-K\"ahler
manifolds with Norden metric}. This is the only class of manifolds
with non-integrable almost complex structure $J$, \ie manifolds
with non-zero Nijenhuis tensor $N$, where
\begin{equation}\label{2.5}
    N(x,y)=\left(\n_x J\right)Jy-\left(\n_y J\right)Jx+\left(\n_{Jx} J\right)y-\left(\n_{Jy}
    J\right)x.
\end{equation}

The associated tensor $N^*$ of $N$ is defined by
\begin{equation}\label{2.6}
    N^*(x,y)=\left(\n_x J\right)Jy+\left(\n_y J\right)Jx+\left(\n_{Jx} J\right)y+\left(\n_{Jy}
    J\right)x.
\end{equation}

The condition \eqref{2.4} is equivalent (\cite{GrMeDj}) to the
condition
\begin{equation}\label{2.7}
    N^*(x,y)=0
\end{equation}
and (\cite{MeMa}) to the condition
\begin{equation}\label{2.8}
\mathop{\s} \limits_{x,y,z} F(Jx,y,z)=0.
\end{equation}

Let $\{e_i\}$ ($i=1,2,\dots,2n$) be an arbitrary basis of the
tangent space of $M$ at a point $p\in M$. The components of the
inverse matrix of $g$, with respect to this basis, are denoted by $g^{ij}$.

Following \cite{GRMa}, the \emph{square norm} $\nJ$ of $\nabla J$
is defined in \cite{MeMa} by
\begin{equation}\label{2.9}
    \nJ=g^{ij}g^{ks}
    g\bigl(\left(\nabla_{e_i} J\right)e_k,\left(\nabla_{e_j}
    J\right)e_s\bigr).
\end{equation}
There, the manifold with $\nJ=0$ is called an
\emph{isotropic-K\"ahler manifold with Norden metric}. It is clear
that every K\"ahler manifold with Norden metric is
isotropic-K\"ahler, but the inverse implication is not always
true.

For quasi-K\"ahler manifolds with Norden metric the following
equality is valid \cite{MeMa}
\begin{equation}\label{2.10}
    \nJ=-2g^{ij}g^{ks}
    g\bigl(\left(\nabla_{e_i} J\right)e_k,\left(\nabla_{e_s}
    J\right)e_j\bigr).
\end{equation}

Let $R$ be the curvature tensor of $\nabla$, \ie $
R(x,y)z=\nabla_x \nabla_y z - \nabla_y \nabla_x z -
    \nabla_{[x,y]}z$ and the corresponding $(0,4)$-tensor is
determined by $R(x,y,z,w)=g(R(x,y)z,w)$. The Ricci tensor $\rho$
and the scalar curvature $\tau$ with respect to $\nabla$ are
defined by
\[
    \rho(y,z)=g^{ij}R(e_i,y,z,e_j),\qquad \tau=g^{ij}\rho(e_i,e_j).
\]

A tensor $L$ of type (0,4) with the pro\-per\-ties%
\be{2.11}%
L(x,y,z,w)=-L(y,x,z,w)=-L(x,y,w,z), \ee \be{2.12} \mathop{\s}
\limits_{x,y,z} L(x,y,z,w)=0 \quad \textit{(the first Bianchi
identity)}
\ee %
is called a \emph{curvature-like tensor}. Moreover, if the
curvature-like tensor $L$ has the property
\begin{equation}\label{2.13}
L(x,y,Jz,Jw)=-L(x,y,z,w),
\end{equation}
we call it a \emph{K\"ahler tensor} \cite{GaGrMi}.

\section{A natural connection on an almost complex manifold with Norden metric}

Let $\nn$ be a linear connection with a tensor $Q$ of the
transformation $\n \rightarrow\nn$ and a torsion tensor $T$, \ie
\[
\nn_x y=\n_x y+Q(x,y),\quad T(x,y)=\nn_x y-\nn_y x-[x,y].
\]
The corresponding (0,3)-tensors are defined by
\begin{equation}\label{3.1}
    Q(x,y,z)=g(Q(x,y),z), \quad T(x,y,z)=g(T(x,y),z).
\end{equation}
The symmetry of the Levi-Civita connection implies
\begin{equation}\label{3.2}
    T(x,y)=Q(x,y)-Q(y,x),
\end{equation}
\begin{equation}\label{3.3}
    T(x,y)=-T(y,x).
\end{equation}

A partial decomposition of the space $\mathcal{T}$
of the torsion (0,3)-tensors $T$(\ie $T(x,y,z)=-T(y,x,z)$) is
valid on an almost complex manifold with Norden metric $(M,J,g)$:
$\mathcal{T}=\mathcal{T}_1\oplus\mathcal{T}_2\oplus\mathcal{T}_3\oplus\mathcal{T}_4$,
where $\mathcal{T}_i$ $(i=1,2,3,4)$ are invariant orthogonal
subspaces \cite{GaMi}. For the projection operators $p_i$ of $\mathcal{T}$ in
$\mathcal{T}_i$ is established:
\begin{equation}\label{3.4}
  \begin{array}{l}
    4p_1(x,y,z)=T(x,y,z)-T(Jx,Jy,z)-T(Jx,y,Jz)-T(x,Jy,Jz),\\[4pt]
    4p_2(x,y,z)=T(x,y,z)-T(Jx,Jy,z)+T(Jx,y,Jz)+T(x,Jy,Jz),\\[4pt]
    8p_3(x,y,z)=2T(x,y,z)-T(y,z,x)-T(z,x,y)-T(Jy,z,Jx)\\[4pt]
    \phantom{p_3(x,y,z)=-}-T(z,Jx,Jy)+2T(Jx,Jy,z)-T(Jy,Jz,x)\\[4pt]
    \phantom{p_3(x,y,z)=-}-T(Jz,Jx,y)+T(y,Jz,Jx)+T(Jz,x,Jy),\\[4pt]
    8p_4(x,y,z)=2T(x,y,z)+T(y,z,x)+T(z,x,y)+T(Jy,z,Jx)\\[4pt]
    \phantom{p_3(x,y,z)=-}+T(z,Jx,Jy)+2T(Jx,Jy,z)+T(Jy,Jz,x)\\[4pt]
    \phantom{p_3(x,y,z)=-}+T(Jz,Jx,y)-T(y,Jz,Jx)-T(Jz,x,Jy).\\[4pt]
  \end{array}
\end{equation}
A linear connection $\nn$ on an almost complex manifold with
Norden metric $(M,J,g)$ is called a \emph{natural connection} if
$\nn J=\nn g=0$. The last conditions are equivalent to $\nn g=\nn
\tilde{g}=0$. If $\nn$ is a linear connection with a tensor $Q$ of
the transformation $\n \rightarrow\nn$ on an almost complex
manifold with Norden metric, then it is  a natural connection iff
the following conditions are valid:
\begin{equation}\label{3.5}
    F(x,y,z)=Q(x,y,Jz)-Q(x,Jy,z),
\end{equation}
\begin{equation}\label{3.6}
    Q(x,y,z)=-Q(x,z,y).
\end{equation}

Let $\Phi$ be the (0,3)-tensor determined by
\begin{equation}\label{3.7}
    \Phi(x,y,z)=g\left(\widetilde{\nabla}_x y - \n_x y,z \right),
\end{equation}
where $\widetilde{\nabla}$ is the Levi-Civita connection of the
associated metric $\tilde{g}$.

\begin{thm}[\cite{GaMi}]\label{t-3.1}
A linear connection with the torsion tensor $T$ on an almost
complex manifold with Norden metric $(M,J,g)$ is natural iff
\begin{equation}\label{3.8}
    4p_2(x,y,z)=g(N(x,y),z)=2\left\{\Phi(z,Jx,Jy)-\Phi(z,x,y)\right\},\\[4pt]
\end{equation}
\begin{equation}\label{3.9}
  \begin{array}{l}
    4p_3(x,y,z)=-\Phi(x,y,z)+\Phi(y,z,x)+\Phi(x,Jy,Jz)\\[4pt]
    \phantom{4p_3(x,y,z)=}+\Phi(y,Jz,Jx)-2\Phi(z,Jx,Jy).
  \end{array}
\end{equation}
\end{thm}

For an arbitrary almost complex manifold with Norden metric
$(M,J,g)$ the following equality is valid \cite{GaGrMi}
\begin{equation}\label{3.10}
    \Phi(x,y,z)=\frac{1}{2}\left\{F(Jz,x,y)-F(x,y,Jz)-F(y,Jz,x)\right\}.
\end{equation}

Applying \eqref{2.4} to \eqref{3.10} we obtain that $(M,J,g)$ is a
quasi-K\"ahler manifold with Norden metric iff
\begin{equation}\label{3.11}
    \Phi(x,y,z)=F(Jz,x,y).
\end{equation}

\begin{thm}\label{thm-3.3}
    For a natural connection with a torsion tensor $T$ on a quasi-K\"ahler manifold with Norden metric
    $(M,J,g)$, which is non-K\"ahlerian, the following properties
    are valid
    \begin{equation}\label{3.12}
        p_2\neq 0,\qquad p_3=0.
    \end{equation}
\end{thm}
\begin{proof}
If we suppose $p_2=0$ then \eqref{3.8} implies $N=0$. Because of
the last con\-di\-tion and \eqref{2.7} the manifold $(M,J,g)$
becomes K\"ahlerian, which is a contradiction. Therefore, $p_2\neq
0$ is valid. From equalities \eqref{3.11}, \eqref{3.9},
\eqref{2.3}, \eqref{2.4} we get $p_3=0$.
\end{proof}

\section{A canonical connection on a quasi-K\"ahler manifold  
with Norden metric}

\begin{defn}[\cite{GaMi}]\label{defn-4.1}
A natural connection with torsion tensor $T$ on an almost complex
manifold with Norden metric $(M,J,g)$ is called a \emph{canonical
connection} if
\begin{equation}\label{4.1}
    T(x,y,z)+T(y,z,x)-T(Jx,y,Jz)-T(y,Jz,Jx)=0.
\end{equation}
\end{defn}

In \cite{GaMi} it is shown that \eqref{4.1} is equivalent to the
condition
\begin{equation}\label{4.2}
        p_1=p_4=0,
\end{equation}
\ie to the condition $T\in\mathcal{T}_2\oplus\mathcal{T}_4$. The
same paper shows that on every almost complex manifold with Norden
metric $(M,J,g)$ there exists an unique canonical connection
$\nn$, and it is determined by
\begin{equation}\label{4.3}
    g(\nn_x y,z)=g(\n_x
    y,z)+\frac{1}{4}\left\{\Phi(x,y,z)-\Phi(z,x,y)-\Phi(Jz,x,Jy)\right\}.
\end{equation}

In \cite{GaMi}, the conditions for the torsion tensor $T$ of the
canonical connection are used to obtain new characteristics for
the eight classes of almost complex manifolds with Norden metric.
The class $\W_0$ of the K\"ahler manifold with Norden metric is
characterized by the condition $T(x,y,z)=0$. The class $\W_3$ of
the quasi-K\"ahler manifold with Norden metric is characterized by
the condition
\begin{equation}\label{4.4}
    T(Jx,y,z)=-T(x,y,Jz).
\end{equation}

The equalities \eqref{3.1}, \eqref{3.3} and \eqref{4.4} for the
torsion tensor $T$ of the canonical connection on a quasi-K\"ahler
manifold with Norden metric imply the properties:
\begin{equation}\label{4.5}
    T(x,y,z)=-T(y,x,z),\quad   T(Jx,y,z)=T(x,Jy,z)=-T(x,y,Jz).
\end{equation}

From \thmref{thm-3.3} and condition \eqref{4.2} we obtain
immediately the following
\begin{thm}\label{thm-4.2}
    For the torsion tensor $T$ of the canonical connection on a quasi-K\"ahler manifold with Norden
    metric the equality $T=p_2$ is valid, \ie
    $T\in\mathcal{T}_2$.
\end{thm}

Equalities \eqref{3.11} and \eqref{4.3} imply the following
\begin{prop}\label{prop-4.3}
    The canonical connection $\nn$ on a quasi-K\"ahler manifold with Norden metric
    $(M,J,g)$ is determined by
\begin{equation}\label{4.7}
    \nn_x y=\n_x y+\frac{1}{4}\left\{\left(\n_y J\right)Jx-\left(\n_{Jy} J\right)x+2\left(\n_{x} J\right)Jy\right\}.
\end{equation}
\end{prop}

Let $\nn$ be the canonical connection on a quasi-K\"ahler manifold
with Norden metric $(M,J,g)$. According to \eqref{4.7}, for the
tensor $Q$ of the transformation $\n \rightarrow\nn$ we have
\begin{equation}\label{4.8}
    Q(x,y)=\frac{1}{4}\left\{\left(\n_y J\right)Jx-\left(\n_{Jy} J\right)x+2\left(\n_{x}
    J\right)Jy\right\}.
\end{equation}
Then
\[
    T(x,y)=\frac{1}{2}\left\{\left(\n_x J\right)Jy+\left(\n_{Jx} J\right)y\right\}.
\]

Hence, having in mind \eqref{2.2} and \eqref{3.1}, we obtain
\begin{equation}\label{4.9}
    T(x,y,z)=\frac{1}{2}\left\{F(x,Jy,z)+F(Jx,y,z)\right\}.
\end{equation}
Substituting $y\leftrightarrow z$ into the above, according to
\eqref{2.3}, we get
\[
    T(x,z,y)=\frac{1}{2}\left\{-F(x,Jy,z)+F(Jx,y,z)\right\}.
\]
Subtracting this from \eqref{4.9} and replacing $y$ with $Jy$ in
the result, we have
\begin{equation}\label{4.10}
    F(x,y,z)=T(x,z,Jy)-T(x,Jy,z).
\end{equation}

The equalities \eqref{4.8}, \eqref{3.1} and \eqref{2.2} imply
\begin{equation}\label{4.11}
    Q(x,y,z)=\frac{1}{4}\left\{F(y,Jx,z)-F(Jy,x,z)+2F(x,Jy,z)\right\}.
\end{equation}
Hence, because of \eqref{2.3} and \eqref{2.4}, we conclude that
\begin{equation}\label{4.12}
    Q(x,y,z)=-Q(y,x,z)+F(Jz,x,y).
\end{equation}

\begin{thm}\label{thm-4.4}
    Let $\tau'$ and $\tau$ be the scalar curvatures for the canonical connection $\nn$ and the Levi-Civita connection $\n$, respectively,
    on a quasi-K\"ahler manifold with Norden metric.
     Then
    \begin{equation}\label{4.13}
        \tau'=\tau-\frac{1}{4}\nJ.
    \end{equation}
\end{thm}
\begin{proof}
According to \eqref{2.1} and \eqref{2.3}, for an almost complex
manifold with Norden metric we have $g^{ij}F(Jz,e_i,e_j)=0$. Then,
from \eqref{4.12}, after contraction by $x=e_i$, $y=e_j$, we
obtain
\begin{equation}\label{4.14}
    g^{ij}Q(e_i,e_j,z)=0.
\end{equation}
Because of $\n g_{ij}=0$ (for the Levi-Civita connection $\n$) and
\eqref{4.14}, we get
\begin{equation}\label{4.15}
    g^{ij}\left(\n_x Q\right)(e_i,e_j,z)=0.
\end{equation}

It is known that for the curvature tensors $R'$ and $R$ of $\nn$
and $\n$, respectively,  the following  is valid:
\[
\begin{split}
&R'(x,y,z,w)=R(x,y,z,w)+ \left(\n_x Q\right)(y,z,w)-\left(\n_y
Q\right)(x,z,w)
\\[4pt]
&\phantom{K(x,y,z,w)=R(x,y,z,w)}+Q\bigl(x,Q(y,z),w\bigr)-Q\bigl(y,Q(x,z),w\bigr).
\end{split}
\]
Then from \eqref{3.6} and \eqref{3.1} it follows that
\begin{equation}\label{4.16}
\begin{split}
&R'(x,y,z,w)=R(x,y,z,w)+ \left(\n_x Q\right)(y,z,w)-\left(\n_y
Q\right)(x,z,w)
\\[4pt]
&\phantom{K(x,y,z,w)=}-g\bigl(Q(x,w),Q(y,z)\bigr)+g\bigl(Q(y,w),Q(x,z)\bigr)
\end{split},
\end{equation}
for an almost complex manifold with Norden metric $(M,J,g)$.

Using a contraction by $x=e_i$, $w=e_j$ in \eqref{4.16} and combining
\eqref{3.6}, \eqref{4.14} and \eqref{4.15}, we find that
the Ricci tensors $\rho'$ and $\rho$ for $\nn$ and $\n$ satisfy
\begin{equation}\label{4.17}
    \rho'(y,z)=\rho(y,z)+g^{ij}\left(\n_{e_i}
    Q\right)(y,z,e_j)+g^{ij}g\bigl(Q(y,e_j),Q(e_i,z)\bigr).
\end{equation}

Similarly, after a contraction by $y=e_k$, $z=e_s$ in \eqref{4.17} and
according to \eqref{4.15}, we obtain
\begin{equation}\label{4.18}
    \tau'=\tau+g^{ij}g^{ks}g\bigl(Q(e_k,e_j),Q(e_i,e_s)\bigr),
\end{equation}
for the scalar curvatures $\tau'$ and
$\tau$ for $\nn$ and $\n$.
The equalities \eqref{4.18} and \eqref{4.8} imply
\begin{equation}\label{4.19}
    g^{ij}g^{ks}g\bigl(Q(e_k,e_j),Q(e_i,e_s)\bigr) =\frac{1}{16}g^{ij}g^{ks}g(P_{jk},P_{si})
\end{equation}
for a quasi-K\"ahler manifold with Norden metric $(M,J,g)$,
where
\[
P_{jk}=\left(\n_{e_j}J\right)Je_k-\left(\n_{Je_j}J\right)e_k+2\left(\n_{e_k}J\right)Je_j.
\]
From \eqref{4.19}, \eqref{2.1}, \eqref{2.9} and \eqref{2.10} we
get
\[
g^{ij}g^{ks}g\bigl(Q(e_k,e_j),Q(e_i,e_s)\bigr)=-\frac{1}{4}\nJ.
\]
The last equality and \eqref{4.18} imply \eqref{4.13}.
\end{proof}

\begin{cor}\label{cor-4.5}
A quasi-K\"ahler manifold with Norden metric is
isotropic-K\"ah\-ler\-ian if and only if the scalar curvatures for
the canonical connection and the Levi-Civita connection are equal.
\end{cor}


\section{A canonical connection with K\"ahler curvature tensor on a quasi-K\"ahler manifold with Norden metric}

The curvature tensor $R'$ of  a natural connection $\nn$ on an
almost complex manifold with Norden metric $(M,J,g)$ satisfies
property \eqref{2.11}, according to \eqref{4.16}. Since $\nn J=0$,
the property \eqref{2.13} is also valid. Therefore, $R'$ is
K\"ahlerian iff the first Bianchi identity \eqref{2.12} is
satisfied. On the other hand, it is known (\cite{KoNo}) that for
every linear connection $\nn$ with a torsion $T$ and a curvature
tensor $R'$ the following equality (the first Bianchi identity) is
valid
\[
\mathop{\s} \limits_{x,y,z} R'(x,y)z=\mathop{\s} \limits_{x,y,z}
\bigl\{\left(\nn_{x} T\right)(y,z)+T(T(x,y),z)\bigr\}.
\]
Since we have $\nn g=0$,  the last equality implies
\[
\mathop{\s} \limits_{x,y,z} R'(x,y,z,w)= \mathop{\s}
\limits_{x,y,z} \bigl\{\left(\nn_{x}
T\right)(y,z,w)+T(T(x,y),z,w)\bigr\}.
\]
Thus, $R'$ satisfies \eqref{2.12} iff
\begin{equation}\label{5.1}
\mathop{\s} \limits_{x,y,z} \bigl\{\left(\nn_{x}
T\right)(y,z,w)+T(T(x,y),z,w)\bigr\}=0.
\end{equation}
This leads to the following
\begin{lem}\label{lem-5.1}
The curvature tensor for the natural connection $\nn$ with a
torsion $T$ on an almost complex manifold with Norden metric is
K\"ahlerian iff \eqref{5.1} is valid.
\end{lem}

We substitute $Jz$ for $z$ and $Jw$ for $w$  in \eqref{5.1}.
Hence, according to \eqref{4.5}, we obtain
\[
\begin{split}
&\left(\nn_{x} T\right)(y,z,w)-\left(\nn_{y}
T\right)(z,x,w)+\left(\nn_{Jz} T\right)(x,y,Jw)\\[4pt]
&+T(T(x,y),z,w)+T(T(y,Jz),x,w)+T(T(Jz,x),y,Jw)=0.
\end{split}
\]

We subtract the last equality from \eqref{5.1}, and substitute
$Jx$ for  $x$ and $Jw$ for $w$ in the result. Then, using
\eqref{4.5}, we get
\begin{equation}\label{5.2}
\begin{split}
&\left(\nn_{z} T\right)(x,y,z)+\left(\nn_{Jz}
T\right)(x,Jy,w)\\[4pt]
&+2T(T(y,z),x,w)+2T(T(z,Jx),y,Jw)=0.
\end{split}
\end{equation}

We substitute $Jy$, $Jz$ for $y$, $z$, respectively, and we apply
\eqref{4.5}. We subtract the obtained equality from \eqref{4.2}
and we reapply \eqref{4.5}. This leads to
\[
T(T(z,x),y,w)=0.
\]
Hence,  \eqref{4.10} and \eqref{2.3} imply
\[
F(Jy,w,T(z,x))=T(y,w,T(z,x)),
\]
and from \eqref{2.2} and \eqref{3.1} we obtain
\begin{equation}\label{5.3}
    g\bigl(T(x,z),T(y,w)-\left(\n_{Jy}J\right)w\bigr)=0.
\end{equation}

Since, according to \eqref{4.9} and \eqref{2.2}, we have
\[
T(y,w)=\frac{1}{2}\big\{\left(\n_{y}J\right)Jw+\left(\n_{Jy}J\right)w
\bigr\},
\]
the following equality is valid
\[
T(y,w)-\left(\n_{Jy}J\right)w=\frac{1}{2}\bigl\{\left(\n_{y}J\right)Jw-\left(\n_{Jy}J\right)w
\bigr\}.
\]
Thus, using \eqref{5.3}, we arrive at the following
\begin{thm}\label{thm-5.2}
Let $(M,J,g)$ be a quasi-K\"ahler manifold with Norden metric,
whose canonical connection has a K\"ahler curvature tensor. Then
the following equality is valid
\[
g\bigl(\left(\n_{x}J\right)Jz+\left(\n_{Jx}J\right)z,\left(\n_{Jy}J\right)w-\left(\n_{y}J\right)Jw
\bigr)=0.
\]
\end{thm}


\section{A canonical connection with parallel torsion on a quasi-K\"ahler manifold with Norden metric}

In this section we consider a canonical connection $\nn$ with
parallel torsion $T$ (\ie $\nn T=0$) on a quasi-K\"ahler manifold
with Norden metric $(M,J,g)$.

According to the Hayden theorem (\cite{Hay})
\begin{equation}\label{6.2}
    Q(x,y,z)=\frac{1}{2}\bigl\{
    T(x,y,z)-T(y,z,x)+T(z,x,y)\bigr\}.
\end{equation}

Combining this with \eqref{3.2}, \eqref{3.5}, \eqref{4.11}, leads
to the following
\begin{prop}\label{prop-6.1}
Let $\nn$ be a natural connection on an almost complex manifold
with Norden metric $(M,J,g)$. Then the tensors $T$, $Q$ and $F$
are parallel or non-parallel at the same time with respect to
$\nn$.
\end{prop}

Let $\nn$ be a natural connection with parallel torsion on an
almost complex manifold with Norden metric $(M,J,g)$. According to
\propref{prop-6.1} we have $\nn Q=0$. Then, having in mind the
formula for the covariant derivative of $Q$, we obtain
\begin{equation}\label{6.3}
    xQ(y,z,w)-Q(\nn_x y,z,w)-Q(y,\nn_x z,w)-Q(y,z,\nn_x w)=0.
\end{equation}
Since $Q$ is the tensor of the deformation $\n\rightarrow\nn$,
applying the formula for the covariant derivative of $Q$ with
respect to $\n$ and equalities \eqref{3.1} and \eqref{3.2}, we
obtain the following
\begin{lem}\label{lem-6.2}
Let $R'$ be the curvature tensor for a natural connection $\nn$
with a parallel torsion $T$ on an almost complex manifold with
Norden metric $(M,J,g)$. Then the following equality is
valid
\begin{equation}\label{6.4}
\begin{split}
    & R'(x,y,z,w)=R(x, y,z,w)+Q(T(x,y),z,w)\\[4pt]
    &\phantom{R'(x,y,z,w)=}+g(Q(y,z),Q(x,w))-g(Q(x,z),Q(y,w)).
\end{split}
\end{equation}
\end{lem}

Let $(M,J,g)$ be a quasi-K\"ahler manifold with Norden metric
whose canonical connection $\nn$ has a parallel torsion $T$. Then,
according to \eqref{3.6}, \eqref{4.12} and \eqref{2.2}, we have
\[
    Q(T(x,y),z,w)=g(Q(z,w),T(x,y))+g(\left(\n_{Jw}J\right)z,T(x,y)).
\]

The last equality and \lemref{lem-6.2} imply
\begin{thm}\label{thm-6.3}
Let $(M,J,g)$ be a quasi-K\"ahler manifold with Norden metric
whose canonical connection $\nn$ has a parallel torsion $T$. Then
for the curvature tensor $R'$ of $\nn$ we obtain
\begin{equation}\label{6.5}
\begin{split}
    & R'(x,y,z,w)=R(x, y,z,w)\\[4pt]
    &\phantom{R'(x,y,z,w)=}+g(Q(y,z),Q(x,w))-g(Q(x,z),Q(y,w))\\[4pt]
    &\phantom{R'(x,y,z,w)=}+g(Q(z,w),T(x,y))+g(\left(\n_{Jw}J\right)z,T(x,y)).
\end{split}
\end{equation}
\end{thm}

Because of \eqref{4.14} we have  $g^{ij}Q(e_i,e_j)=0$. Then, from
\eqref{6.5} via a contraction by $x=e_i$, $w=e_j$, we get
\begin{equation}\label{6.6}
\begin{split}
    &\rho'(y,z)=\rho(y,z)-g^{ij}g(Q(e_i,z),Q(y,e_j))\\[4pt]
    &\phantom{\rho'(y,z)=}+g^{ij}g(Q(z,e_j),T(e_i,y))+g^{ij}g(\left(\n_{Je_j}J\right)z,T(e_i,y)),
\end{split}
\end{equation}
where $\rho'$ and $\rho$ are the Ricci tensors for $\nn$ and $\n$,
respectively.

Combining \eqref{3.1}, \eqref{4.12}, \eqref{4.5}, \eqref{2.8},
\eqref{3.2}, \eqref{2.3} and \eqref{2.2}, we obtain
\begin{equation}\label{6.7}
\begin{split}
    &g(Q(z,e_j),T(e_i,y))=g(Q(e_j,z),Q(y,e_i))-g(Q(e_j,z),Q(e_i,y))\\[4pt]
    &\phantom{g(Q(z,e_j),T(e_i,y))=}-g(\left(\n_{Je_j}J\right)z,T(e_i,y))-g(\left(\n_{Jz}J\right)e_j,T(e_i,y)).
\end{split}
\end{equation}

We get the following equality from \eqref{6.6} and \eqref{6.7}:
\begin{equation}\label{6.8}
    \rho'(y,z)=\rho(y,z)-g^{ij}g(Q(e_j,z),Q(e_i,y))-g^{ij}g(\left(\n_{Jz}J\right)e_j,T(e_i,y)).
\end{equation}
A contraction by $y=e_k$, $z=e_s$ leads to
\begin{equation}\label{6.9}
    \tau'=\tau-g^{ij}g^{ks}g(Q(e_j,e_s),Q(e_i,e_k))-g^{ij}g^{ks}g(\left(\n_{Je_s}J\right)e_j,T(e_i,e_k)),
\end{equation}
where $\tau'$ and $\tau$ are the respective scalar curvatures for
$\nn$ and $\n$.

Using \eqref{4.7}, \eqref{2.10} and \eqref{2.9}, we get
\begin{equation}\label{6.10}
    g^{ij}g^{ks}g(Q(e_j,e_s),Q(e_i,e_k))=-\frac{1}{2}\nJ.
\end{equation}

From  \eqref{2.10} and
$2T(e_i,e_j)=\left(\n_{e_i}J\right)Je_k+\left(\n_{Je_j}J\right)e_k$
 we have
\begin{equation}\label{6.11}
    g^{ij}g^{ks}g(\left(\n_{Je_s}J\right)e_j,T(e_i,e_k))=\frac{1}{4}\nJ.
\end{equation}
Then, \eqref{6.9}, \eqref{6.10} and \eqref{6.11} imply
\[
\tau'=\tau+\frac{1}{4}\nJ.
\]

From the last equality and \eqref{4.13} we obtain the following
\begin{thm}\label{thm-6.4}
Let $(M,J,g)$ be a quasi-K\"ahler manifold with Norden metric
whose canonical connection $\nn$ has a parallel torsion $T$. Then
$(M,J,g)$ is iso\-tropic-K\"ahlerian.
\end{thm}


\section{A relation between the canonical connection, the $B$-connection, and the $KT$-connection on a quasi-K\"ahler manifold with Norden metric}

Let $(M,J,g)$ be a quasi-K\"ahler manifold with Norden metric. Let
us consider the following connections on $(M,J,g)$: the canonical
connection $\n^C$, the $B$-connection $\n^B$ (\cite{Mek-1}) and
the connection $\n^{KT}$ with a totally skew-symmetric torsion
(\cite{Mek-2}).

Let
\[
\n^B=\n+Q^B,\quad\n^{KT}=\n+Q^{KT},\quad\n^C=\n+Q^C.
\]
Then, according to \cite{Mek-1,Mek-2} and \eqref{4.11}, we have
\[
\begin{split}
    &Q^B(x,y,z)=\frac{1}{2}F(x,Jy,z),\\[4pt]
    &Q^{KT}(x,y,z)=-\frac{1}{4}\mathop{\s} \limits_{x,y,z} F(x,y,Jz),\\[4pt]
    &Q^{C}(x,y,z)=\frac{1}{4}\bigl\{F(Jz,x,y)-F(x,y,Jz)-F(Jy,x,z)\bigr\}.\\[4pt]
\end{split}
\]
After that, using \eqref{2.3} and \eqref{2.4}, we obtain
\[
Q^B=\frac{1}{2}\left(Q^{KT}+ Q^{C}\right).
\]

This leads to
\begin{prop}
The $B$-connection on a quasi-K\"ahler manifold with Norden metric
is the mean connection for the canonical connection and the
$KT$-connection.
\end{prop}


\bigskip
\noindent \textsc{Dimitar Mekerov}\\
University of Plovdiv, Faculty of Mathematics and
Informatics\\
236 Bulgaria Blvd., 4003 Plovdiv, Bulgaria\\
mircho@uni-plovdiv.bg\\

\end{document}